\documentclass[a4paper]{article}

\usepackage[utf8]{inputenc}
\usepackage{lmodern}
\usepackage[title]{appendix}
\usepackage{amssymb,amsfonts,amsthm,amsmath,bbm,mathrsfs,aliascnt,mathtools}
\usepackage[english]{babel}
\usepackage[colorlinks]{hyperref}
\usepackage{tikz}
\usetikzlibrary{decorations}
\usetikzlibrary{decorations.pathreplacing}
\tikzset{individu/.style={draw,thick}}

\theoremstyle{plain}
\newtheorem{theorem}{Theorem}[section]
\newtheorem{corollary}[theorem]{Corollary}
\newtheorem{lemma}[theorem]{Lemma}

\theoremstyle{definition}

\theoremstyle{remark}
\newtheorem{remark}[theorem]{Remark}

\numberwithin{equation}{section}

\usepackage{tikz}

\makeatletter \newcommand \listoftodos{\section*{Todo list} \@starttoc{tdo}}
\newcommand\l@todo[2]
{\par\noindent \textit{#2}, \parbox{10cm}{#1}\par} \makeatother

\newcommand{\N}{\mathbb{N}}
\newcommand{\Z}{\mathbf{Z}}
\newcommand{\R}{\mathbb{R}}
\newcommand{\W}{\mathbf{W}}

\newcommand{\calD}{\mathcal{D}}
\newcommand{\calM}{\mathcal{M}}

\newcommand{\NN}{\mathbf{N}}

\renewcommand{\S}{\mathbf{S}}

\newcommand{\ind}[1]{\mathbf{1}_{\left\{#1\right\}}}
\newcommand{\indset}[1]{\mathbf{1}_{#1}}

\newcommand{\crochet}[1]{{\langle #1 \rangle}}
\renewcommand{\bar}[1]{\overline{#1}}
\renewcommand{\tilde}[1]{\widetilde{#1}}

\newcommand{\e}{\mathrm{e}}
\newcommand{\dd}{\mathrm{d}}
\newcommand{\egaldistr}{{\overset{(d)}{=}}}

\DeclareMathOperator{\E}{\mathbb{E}}

\renewcommand{\P}{\mathbb{P}}

\newcommand{\Mf}{\mathcal{M}_f}

\newcommand{\x}{\mathbf{x}}
\newcommand{\X}{\mathbf{X}}
\newcommand{\U}{\mathbf{U}}
\newcommand{\y}{\mathbf{y}}
\newcommand{\z}{\mathbf{z}}

\title{Scaling limits of branching random walks\\ and
branching stable processes}
\author{Jean Bertoin\thanks{Institute of Mathematics, University of Zurich, Switzerland.} \and Hairuo Yang  \thanks{Institute of Mathematics, University of Zurich, Switzerland.}}
\date{\today}

\begin{document}

\maketitle

\begin{abstract}  Branching-stable processes have recently appeared as counterparts of stable subordinators,
when addition of real variables is replaced by branching mechanism for point processes. 
Here, we are interested in their domains of attraction and describe explicit conditions for a branching random walk to converge
after a proper magnification to a branching-stable process. This contrasts with deep results that have been obtained during the last decade
 on the asymptotic behavior of branching random walks and which involve either shifting without rescaling, or demagnification. \end{abstract}

\medskip


\section{Introduction}
\label{sec:introduction}
We start by recalling the classical stable  limit theorem in the special case of nonnegative random variables, referring to the treatises \cite{GK} and \cite{IL} for the complete story.
Consider a random variable $Y\geq 0$ whose tail distribution $\bar F(y)=\P(Y>y)$ is regularly varying at infinity with index $-\beta$ for some $\beta\in(0,1)$, that is
$\lim_{y\to \infty} \bar F(ay)/\bar F(y) = a^{-\beta}$ for all $a>0$. Let also $Y_1, Y_2, \ldots$ denote a sequence of i.i.d. copies of $Y$.
Then for any sequence $(a_n)_{n\in \N}$ of positive real numbers such that $\lim_{n\to \infty} n\bar F(a_n)=1$ (as a consequence,  $(a_n)$ varies regularly with index $1/\beta$),  
the sequence $(Y_1+\ldots+Y_n)/a_n$  of rescaled partial sums converges in distribution as $n\to \infty$ to  a stable law  on $\R_+$ with exponent $\beta$. 
The purpose of this paper is to present an analog of this stable limit theorem in the setting of branching random walks. A rather surprising feature is that the counterpart of the index $\beta$, denoted here by $-\alpha$, can then  be any negative real number. 

The study of the asymptotic behavior of branching processes has attracted a lot of attention and efforts during many years. First, for light tailed displacements, let us merely single out  \cite{Big77, Big92} and \cite{BramA, BramB}  amongst the most important earlier contributions. More recently, \cite{Aidek} established a remarkable limit theorem in distribution for the minimum of a super-critical branching random walk. Then  \cite{ABBS} and \cite{ABK1, ABK2}  
showed that the point process formed by a branching Brownian motion seen from its left-most atom converges in distribution  to   a random counting measure that can be constructed from some marked  Poisson point process and is often called decorated  Poisson point process. Finally, \cite{Madaule} obtained the counterpart  of the Brownian results in the framework of branching random
walks. For heavy tailed displacements,  weak limit theorems have been established in \cite{Durr} for the location of the rightmost particle, and in \cite{BHR1, BHR2} for the whole point process. We shall recall  below some of these results, a bit informally for the sake of simplicity. 

Let us introduce beforehand some notation that we shall use throughout this text. 
For every $b>0$ and every  counting measure $m=\sum_j \delta_{x_j}$ on a vector space, we write $bm$ for the push-forward of $m$ by the dilation  with factor $b$, $x\mapsto bx$, that is 
$$bm=\sum_j \delta_{bx_j}. $$
In other words,  we shall think of a counting measure as a multiset of atoms repeated according to their multiplicities, and in this setting, multiplication by a scalar acts on the location of atoms rather than on the measure itself. 
We also write $\crochet{m, f}$ for the integral of a function $f$ with respect to a counting measure $m$, that is 
$$\crochet{m, f}=\sum_j f(x_j),$$ 
whenever this makes sense. In particular, the mass $m(B)$ of a set $B$ is written as 
$\crochet{m, \indset{B}}$, with $\indset{B}$ denoting the indicator function of $B$. 

Assume that $(\Z(n))_{n\geq 0}$ is a branching random walk on $\R$, where $\Z(n)$ is the point process induced by the locations of the particles at generation $n$; suppose also for simplicity that $\Z(0)=\delta_0$ is the Dirac point mass at $0$. We refer to the lecture notes \cite{Shi}  for the necessary background, terminology,  and of course, much more.
The fundamental assumption in  \cite{Madaule} is  that for the functions $\mathbf 1(x)=1$, $f(x)=\e^{-x}$, $g(x)=x\e^{-x}$ and $h(x)=x^2\e^{-x}$:
\begin{align}\label{E:fundhyp} 
& \E(\crochet{\Z(1),\mathbf 1})>1, \ \E(\crochet{\Z(1),h})<\infty,\nonumber \\
& 
 \E(\crochet{\Z(1),f })=1 \ \text{and } \E(\crochet{\Z(1),g})=0. \end{align}
This may look stringent; however in practice a simple linear map transforms many branching random walks into another branching random walk that satisfies \eqref{E:fundhyp}. Taking also for granted some further technical requirements, the point process obtained by  shifting the atoms of $\Z(n)$ by $-\frac{3}{2}\log n -\log D_{\infty}$, where $D_{\infty}$ denotes the terminal value of the so-called derivative martingale,  then converges in distribution as $n\to \infty$.  It is remarkable that this weak limit theorem involves shifting, but not rescaling. Moreover, the limiting point process can be described as a  decorated Poisson point process.

More recently,  \cite{BHR1} considered branching random walks obtained by superposing i.i.d. heavy tailed displacements to a supercritical Galton-Watson tree. Specifically, one supposes there that the first generation has the form 
$$\Z(1)=\sum_{j=1}^{N} \delta_{Y_j},$$ where $Y_1, \ldots$ is a sequence of i.i.d. real random variables and $N$ an independent integer valued random variable with $\E(N)=\mu>1$ and $\E(N \log N)<\infty$.  Assuming further that the tail distribution $\P(|Y_1|>x)$ is regularly varying at infinity with index $-\beta<0$ and a tail balanced condition, there is a sequence $(b_n)$ of positive real numbers that grows roughly like $\mu^{n/\beta}$ such that,  conditionally on non-extinction, $b_n^{-1}\Z(n)$ converges weakly to a so-called Cox cluster process. See Theorem 2.1 in \cite{BHR1} for a precise statement. The same authors extended their result  and replaced the assumption that the sequence $Y_1, \ldots$ is i.i.d. by a weaker condition involving regular variation in the sense of \cite{HL,  LRR}; see Theorem 2.6 in \cite{BHR2}.

We shall now present, again a bit informally, the main result of the present work. 
We consider henceforth a branching random walk $(\Z(n))_{n\geq 0}$ on the nonnegative half-line 
$\R_+=[0,\infty)$, and assume as before that $\Z(0)=\delta_0$. We suppose that $\Z(1)$ has a single atom at the origin a.s. and we write 
$$0< X_1\leq X_2 \leq \ldots \leq \infty$$
for the sequence of atoms of $\Z(1)$ on $(0,\infty)$, ranked in the non-decreasing order and repeated according to their multiplicities, with the convention that $X_j=\infty$ if and only if
$\Z(1)$ has less than $j$ atoms in $(0,\infty)$. In other words,
$$\Z(1)= \delta_0 + \sum_{j\geq 1} \delta_{X_j},$$
where we implicitly agree to discard atoms at $\infty$ in the series on the right-hand side. 
Strictly speaking, after each unit of time, every individual dies and simultaneously gives birth to children among those a single one
occupies the same location as its parent. If we rather view this child as the same individual as its parent (which thus survives after reaching age $1$), we may think of this branching random walk as describing a spatial population model with static immortal individuals, such that at each generation, each individual gives birth to children located at its right and at distances given by independent copies of $X_1, X_2, \ldots$. 

Under these  assumptions, the log-Laplace transform of the intensity measure of $\Z(1)$,
$$\psi(t)=\log \E(\crochet{\Z(1), \e^{-t\bullet}})=\log\left( 1+\sum_{j\geq 1} \E(\e^{-tX_j})\right)\in (0,\infty], $$ 
fulfills  $\psi(t)>0$ and $\psi'(t)<0$ for all $t\geq 0$ in the domain of $\psi$. As a consequence, no linear transform of $\Z(1)$ can satisfy \eqref{E:fundhyp}. 

We next introduce the three assumptions on the point process $\Z(1)$ under which we shall establish a scaling limit theorem for the branching random walk.
The first is that,  if $F_1(t)= \P(X_1\leq t)$ denotes the cumulative distribution function of the first positive atom $X_1$, then  for some $\alpha>0$:
\begin{equation} \label{E:C1} F_1\text{  is regularly varying at $0+$ with index $\alpha$}.
\end{equation}
The second is the existence of a scaling limit for the conditional distribution of  
$$\Z^*(1)=\Z(1)-\delta_0=\sum_{j\geq 1} \delta_{X_j}$$
given that  $X_1$ is small.
Specifically recall  the notation $bm$ for the push-forward of a measure $m$ by the dilation with factor $b$, and view $\Z^*(1)$ as a random variable on the space $\mathcal M$ of locally finite counting measures on $\R_+$ endowed with the vague topology (see e.g. Appendix A.2 in \cite{Kall}). Our second assumption is:
\begin{eqnarray} \label{E:C2}
&\text{ the conditional law of } t^{-1}\Z^*(1) 
 \text{ given } X_1\leq t\nonumber \\ 
& \text{ has a weak limit as $t\to 0+$.} 
\end{eqnarray}
Our final assumption is that the log-Laplace transform $\psi$ of the intensity mesure of $\Z(1)$ fulfills
\begin{equation} \label{E:C3}
\sup_{n\geq 1}  n\ \psi(1/a_n)<\infty,
\end{equation}
where  $(a_n)_{n\geq 1}$ is a sequence of positive real numbers such that
$$\lim_{n\to \infty} nF_1(a_n)=1.$$
Note that $(a_n)$ is then regularly varying with index $-1/\alpha$.

Our first two assumptions are given explicitly in terms of the point process $Z(1)$; however the interpretation of the third may be less clear. One has to recall that the function $n\psi$ is the 
log-Laplace transform of the intensity measure of the branching random walk at the $n$-th generation,
in particular 
$$\E(\crochet{a_n^{-1}\Z(n), \e^{-\bullet}})=\E(\crochet{\Z(n), \e^{-a_n^{-1}\bullet}})=\exp(n\psi(1/a_n)).$$
Thus \eqref{E:C3} should be view as a natural condition to ensure that on average, the rescaled branching random walk $a_n^{-1}\Z(n)$ remains locally bounded. We also refer to the forthcoming Remark \ref{R:1}
for further comments on these assumptions. 

Under these assumptions,
the sequence of rescaled processes in continuous time
$$\left( a_n^{-1} \Z(\lfloor nt\rfloor)\right)_{t\geq 0}$$
converges  in distribution as $n\to \infty$, on a space of rcll (right continuous with left limits) functions with values on a certain space of counting measures. The limit is a branching-stable process $\mathbf S=(\mathbf S(t))_{t\geq 0}$ introduced in  \cite{BCM}. In words, $\S$ is a branching process in continuous time which is self-similar with scaling exponent $-\alpha<0$, in the sense that for every $c>0$, there is the identity in distribution
$$ (\mathbf S(ct))_{t\geq 0} \,\egaldistr\,  (c^{-1/\alpha} \mathbf S(t))_{t\geq 0}.$$
The law of $\mathbf S$ is characterized by the exponent $\alpha$ and the limiting distribution appearing in \eqref{E:C2}.

Although our result bears somehow the same flavor as those in \cite{BHR1, BHR2} that have been mentioned above (notably our assumptions \eqref{E:C1} and \eqref{E:C2} resemble the hypothesis of regular variation for the distribution of $\Z(1)$  in \cite{BHR2}), there are major differences. The most obvious one is that \cite{BHR1, BHR2} work with a demagnification $b_n^{-1} \Z(n)$ with $b_n\approx c^n\gg 1$, whereas, at the opposite, we consider here a magnification $a_n^{-1}\Z(n)$ with $a_n\approx n^{-1/\alpha} \ll 1$.  Roughly speaking,  extreme value theory and the so-called principle of a single big jump (see \cite{Durr}) lie at the heart of the approach in  \cite{BHR1, BHR2}, whereas our result rather depends on Markov chain approximations of Feller processes. Another significant  difference is that we obtain a weak limit theorem for processes depending on time, whereas \cite{BHR1, BHR2} consider the branching random walk $\Z$ at one given generation $n$ only. Last but not least, branching stable processes with negative indices only exist in the one-sided framework (i.e. on a half-line, see Lemma 2.2 in \cite{BCM}), and hence one should not expect a two-sided version as in  \cite{BHR1, BHR2}.

The rest of this article is organized as follows. Section 2 is devoted to preliminaries. We shall first provide some background on a family of branching processes in continuous time which have been introduced by Uchiyma \cite{Uchi}, and point out that under appropriate assumptions, these arise as the weak limits of certain families of branching random walks in discrete time. Then, we shall recall some features on branching stable processes and their trimmed versions, and show that the latter belong to the family considered by Uchiyama. Our main result will then be properly stated and proved in Section 3. We shall need to work with various  spaces of counting measures, and  for the reader's convenience, we gather their definition and notation in an appendix.

\section{Preliminaries}

\subsection{Weak convergence to Uchiyama's branching process} \label{S:Uch}
Even though this work is mainly concerned with branching processes living on the positive half-line,  in this section, we shall consider a bit more generally the $d$ dimensional setting.

Uchiyama \cite{Uchi} introduced a family of branching processes that can be thought of as analogs of branching random walks in continuous time; they can be described informally as follows. Fix $r>0$ and let $\Pi$ denote a probability measure on the space of finite counting measures on $\R^d$. 
We write $r\cdot \Pi$ for the ordinary scalar multiplication of the measure $\Pi$ by $r$ 
(to avoid a possible confusion with the notation $bm$ defined in the Introduction); in particular the total mass of $r\cdot \Pi$ equals $r$. 
Imagine a particle system with no interactions on $\R^d$, where each particle, say located at $x$, dies at  rate $r$ and does not move during its lifetime. At the time of its death, it gives birth to children whose locations relative to $x$ are given by a point process distributed according to $\Pi$, independently of the other particles. The process $\U=(\U(t))_{t\geq 0}$
which records the locations of particles alive as a function of time, is a branching process considered by Uchiyama. 
The finite measure $r \cdot \Pi$ on the space of finite counting measures  characterizes the evolution of $\U$; it will be referred to as the reproduction rate.
The purpose of this section is to point out that $\U$ arises as the weak limit of  certain sequences of branching random walks, 
much in the same way as compound Poisson processes appear as weak limits of certain sequences of  discrete time random walks with rare non-zero steps. In this direction, we shall first describe $\U$
as a Feller process and determine its infinitesimal generator. 

We write $\Mf$ for the set of finite counting measures on $\R^d$, endowed with the L\'evy-Prokhorov distance. A sequence $(\x_n)_{n\in \N}$  converges to $\x$ in $\Mf$ if and only if 
$\lim_{n\to \infty}\crochet{\x_n,f}=\crochet{\x,f}$ for every $f\in  \mathcal{C}_b(\R^d)$  (i.e. $f:\R^d \to \R$ is continuous and bounded). It is seen from Prokhorov's theorem that a subset $\mathcal{S}\subset \Mf$
is relatively compact if and only if both, the total mass remains bounded, viz.
$$\sup_{\x\in \mathcal{S}} \crochet{\x,\mathbf{1}} < \infty,$$
and there are no atoms out of some compact subset of $\R^d$, that is
there exists some $b>0$ such that
$$\crochet{\x,\mathbf{1}_{B_b^c}}=0\quad \text{ for all }\x\in \mathcal{S},$$ 
where $B_b^c$  denotes  the complement of the closed ball $B_b=\{x\in \R^d: |x|\leq  b\}$. Hence $\Mf$ is a locally compact metric space with one-point compactification $\bar \Mf=\Mf\cup\{\partial\}$,
 and a sequence $(\x_n)_{n\in \N}$ in $\Mf$ converges to $\partial$ as $n\to \infty$ if and only if either 
 $$\lim_{n\to \infty}  \crochet{\x_n,\mathbf{1}}=\infty,$$
 or 
 $$\liminf_{n\to \infty} \crochet{\x_n,\mathbf{1}_{B_b^c}}\geq 1 \quad \text{ for all }b>0.$$
 We also write ${\mathcal C}_0(\Mf)$ for the space of continuous maps $\varphi: \Mf \to \R$ with $\lim_{\x\to \partial}\varphi(\x)=0$. 
 
 A random variable with values in $\Mf$ is called a finite point process.
 Recall that $\Pi$ is a probability measure on $\Mf$ which determines the statistics of the point processes induces by the offsprings in $\U$. We assume throughout this section that
 no particles die without offspring and the number of children has a finite expectation, that is
 \begin{equation}\label{E:finiteint} \Pi(\crochet{\x,\mathbf 1}=0)=0 \quad \text{and} \quad \int_{\Mf} \crochet{\x,\mathbf{1}}  \Pi(\dd \x)< \infty.
 \end{equation}
The process $\crochet{\U(t),\mathbf{1}}$ that counts the number of particles alive at time $t\geq 0$, is a one-dimensional continuous time Markov branching process in the sense of Chapter III in \cite{AthNey}, and \eqref{E:finiteint} ensures that $\crochet{\U(\cdot),\mathbf{1}}$ remains finite (no explosion). In particular, this enables us to view 
$\U$ as a Markov process with values in $\Mf$. In order to analyze its semigroup, we need to introduce a few notation.

First, for every $\x\in \Mf$ and $y\in \R^d$, we write $y+\x$ for the finite counting measure obtained by translating each and every atom of $\x$ by $y$. Equivalently, 
$y+\x$ is the  pushforward measure of $\x$ by the translation $x\mapsto y+x$; in particular $y+\delta_x=\delta_{x+y}$. 
The map $(y,\x)\mapsto y+\x$ is continuous from $\R^d\times \Mf$ to $\Mf$.
Next, for any finite sequence $\x^1, \ldots, \x^k$ in $\Mf$, we write 
$\x^1 \sqcup \cdots \sqcup \x^k$ for the sum of those counting measures, so that the family of atoms of $\x^1 \sqcup \cdots \sqcup \x^k$ is the multiset which results from the superposition of the families of atoms of $\x^1, \ldots,  \x^k$. This enables us to express the branching property of $\U$ as follows. 
Consider a finite counting measure $\x=\sum_{j=1}^k \delta_{x_j}$ with atoms $x_1, \ldots, x_k\in \R^d$. Let $\U^1, \ldots, \U^k$ be independent copies of $\U$, all started from the Dirac point mass at the origin. Then the process 
$$(x_1+\U^1(t))\sqcup \cdots \sqcup (x_k+\U^k(t)), \qquad t\geq 0$$
is a version of $\U$ started from $\x$.
Recall that $r>0$ is the rate of death of particles; we can now state:

\begin{lemma}\label{L1} A Uchiyama branching process $\U$ with reproduction rate $r\cdot \Pi$ satisfying \eqref{E:finiteint} is a Feller process on $\Mf$. Its infinitesimal generator 
${\mathcal A}$ has full domain ${\mathcal C}_0(\Mf)$ and is given for every $\x=\sum_{j=1}^k\delta_{x_j} \in \Mf$ and $\varphi\in {\mathcal C}_0(\Mf)$ by
$${\mathcal A}\varphi(\x)= r \sum_{j=1}^k\int_{\Mf} \varphi\left(  \x^*_j  \sqcup (x_j+\y)
\right) \Pi(\dd \y) - rk\varphi(\x),$$
where $ \x^*_j =\sum_{i\neq j} \delta_{x_i}$ denotes the counting measure obtained by removing the atom $x_j$ from $\x$.
\end{lemma}
\begin{remark} The first assumption in \eqref{E:finiteint} that each particle has a non-empty offspring is crucial, and the Feller property always fails otherwise. To see why, let $\varphi$ denote the indicator function of the zero measure $\varnothing$  (no atom), which is a continuous function on $\Mf$ with compact support. Consider also a sequence $(x_n)$ in $\R^d$ which tends to $\infty$, so $\delta_{x_n}$ tends to $\partial$ in $\Mf$.  Clearly, if the probability that 
a particle dies without progeny is strictly positive, then 
$$\E(\varphi(\U(1))\mid \U(0)=\delta_{x_n})=\P(\U(1)=\varnothing\mid  \U(0)=\delta_{x_n})$$ does not converge to $0$ as $n\to \infty$. 
\end{remark}

\begin{proof}   Let $\x=\sum_{j=1}^k\delta_{x_j}$ be a finite counting measure. In the notation introduced before the statement, the probability that $\U^1(t)=\ldots = \U^k(t) = \delta_0$ tends to $1$ as $t\to 0+$.
Therefore 
$$\lim_{t\to 0+} (x_1+\U^1(t))\sqcup \cdots \sqcup (x_k+\U^k(t))=\x \qquad \text{in probability},$$
and for any function $\varphi\in  {\mathcal C}_0(\Mf)$, we have 
$$\lim_{t\to 0+} \E\left( \varphi\left((x_1+\U^1(t))\sqcup \cdots \sqcup (x_k+\U^k(t))\right) \right) = \varphi(\x).$$

Next, consider any sequence $(\x_n)_{n\in \N}$ in $\Mf$ which converges to $\x$. Recall that for $n$ sufficiently large, the number of atoms $\crochet{ \x_n,\mathbf 1}$ of $\x_n$ coincides with $k$, and then we  can enumerate those atoms, i.e. we can write $\x_n=\sum_{j=1}^k\delta_{x_{n,j}}$, in such a way that the sequence $((x_{n,1}, \ldots, x_{n,k}))_{n\in \N}$ converges to $(x_1, \ldots, x_k)$ in $\R^{d\times k}$ as $n\to \infty$. Since  $\varphi$ is continuous and bounded, we  easily conclude that the map
$$\x \to \E\left( \varphi\left(( (x_1+\U^1(t))\sqcup \cdots \sqcup (x_k+\U^k(t))\right) \right)$$
is continuous on $\Mf$. To check that this map has also limit $0$ as $\x\to \partial$, since $\varphi$ is bounded and has limit $0$ at $\partial$,  we simply need to verify that 
$$\lim_{\x\to \partial} (x_1+\U^1(t))\sqcup \cdots \sqcup (x_k+\U^k(t)) = \partial \qquad \text{in probability.}$$

So let $(\x_n)_{n\in \N}$ be a sequence in $\Mf$ which tends to $\partial$. In the case where the total mass of $\x_n$ goes to infinity,  the assumption that 
no particles die without offspring makes the above claim obvious. In the opposite case,  there exists $k\geq 1$ and a subsequence $(\x'_n)_{n\in \N}$ extracted from $(\x_n)_{n\in \N}$
with $\crochet{\x'_n,\mathbf 1}=k$ for all $n$. Since $\lim_{n\to \infty}\x'_n=\partial$, the largest atom of $\x'_n$ tends to $\infty$ as $n\to \infty$, and the claim above follows again
from the fact that $\U(t)\neq \varnothing$ a.s. 

This completes the proof of the Feller property. 
The formula for the infinitesimal generator ${\mathcal A}$ should then be plain from the dynamics of Uchiyama branching processes and the alarm clock lemma.
\end{proof}

We write $\calD(\Mf)$ for the space of rcll  functions $\omega: \R_+\to \Mf$, which we endow with the Skorohod $J_1$-topoplogy (we refer e.g. to Appendice A2 of \cite{Kall} for quick background). 
The Feller property ensures the existence of a version of $\U$ with sample paths in $\calD(\Mf)$ a.s., and we shall henceforth always work with such a version. 

We now arrive at the main purpose of this section, namely the observation that  $\U$  arises as the weak limit of certain sequences of branching random walks on $\R^d$. 
\begin{lemma} \label{L2} Let $r>0$ and  $\Pi$ a probability measure on $\Mf$ satisfying \eqref{E:finiteint}. 
For each $n\in\N$, let $\Z^n=(\Z^n(k))_{k\in \N}$ be a branching random walk on $\R^d$ started from $\delta_0$. 
Assume that:
\begin{enumerate}
\item $\crochet {\Z^n(1), \mathbf 1}\geq 1$ a.s. and $\E(\crochet {\Z^n(1), \mathbf 1})<\infty$ for each $n\in \N$,
\item $\P(\Z^n(1)\neq \delta_0) \sim r/n$ as $n\to \infty$,
\item $\lim_{n\to \infty} \P(\Z^n(1)\in \cdot \mid \Z^n(1)\neq \delta_0)= \Pi(\cdot)$ in the sense of weak convergence for distributions on $\Mf$.
\end{enumerate}
Then we have
$$\lim_{n\to \infty} \left(\Z^n(\lfloor tn \rfloor)\right)_{t\geq 0} = (\U(t))_{t\geq 0}$$
in the sense of  weak convergence on $\calD(\Mf)$, where in the right-hand side, $\U$ denotes a Uchiyama branching process with reproduction rate $r\cdot \Pi$ started from $\delta_0$.
\end{lemma}
\begin{proof}  We shall establish the claim by verifying that the conditions of a well-known convergence theorem for Markov chains are satisfied for a killed version of the processes, where the killing occurs at the time when the total mass exceeds some  fixed threshold. Then letting this threshold tend to infinity will complete the proof.

To start with, the first assumption ensures that every particle in $\Z^n$ has at least one child at the next generation and that $\crochet {\Z^n(k), \mathbf 1}<\infty$ a.s. for all $k\in \N$. In particular, the branching random walks $\Z^n$ can be thought of as Markov chains with values in $\Mf$. 
Fix some $\ell\geq 1$ and write 
$$\mathcal P_{\ell}=\{\x\in \Mf: 
\crochet{\x, \mathbf 1}\leq \ell\}$$
 for the open subset of counting measures with at most $\ell$ atoms.
The map 
$$K_{\ell}: \Mf \to \bar {\mathcal{P}}_{\ell}=\mathcal{P}_{\ell} \cup\{\partial\}, \quad  K_{\ell}(\x)=\left\{
\begin{matrix}
\x & \text{ if }\crochet {\x, \mathbf 1}\leq \ell, \\
\partial & \text{ otherwise,}
\end{matrix}
\right.
$$
is continuous. 

Since the total mass $\crochet {\Z^n(k), \mathbf 1}$ is non-decreasing in the variable $k\geq 0$ a.s., $K_{\ell}(\Z^n(\cdot))$ describes the branching random walk $\Z^n(\cdot)$ killed when its total mass exceeds $\ell$, and is also a Markov chain. Similarly, $K_{\ell}(\U(\cdot))$ 
 still a Feller process (recall that $K_{\ell}$ is continuous) on the compact metric space $\bar{\mathcal{P}}_{\ell}$. We deduce from Lemma \ref{L1} that its infinitesimal generator $\mathcal{A}_{\ell}$ has full domain
 $\mathcal{C}(\bar{\mathcal{P}}_{\ell})$ and is given by
 $$\mathcal{A}_{\ell}\varphi(\x) = 
 \mathcal{A}(\varphi\circ  K_{\ell})(\x) \quad 
 \text{for }\x\in \mathcal{P}_{\ell}
 \text{ and } \mathcal{A}_{\ell}\varphi(\partial)=0.
 $$
 
Let $\x\in \mathcal P_{\ell}$ with $ \x=\sum_{j=1}^k\delta_{x_j}$ for some $ k\leq \ell$, and let  $\z^n_1, \ldots, \z^n_k$ be i.i.d. copies of $\Z^n(1)$. By the branching property,
the point process 
$$\boldsymbol{\zeta}^n(\x)=(x_1+\z^n_1)\sqcup \cdots \sqcup (x_k+\z^n_k)$$
 has the distribution of $\Z^n(1)$ started from $\x$. 
 The events $\{\z^n_j\neq \delta_0\}$ for $j=1, \ldots ,k$ are independent, and by the second assumption of the statement, each has probability $\sim r/n$.
 
Consider  any  $\varphi\in  \mathcal{C}(\bar{\mathcal{P}}_{\ell})$ and 
recall  that  $\varphi$ is uniformly continuous since $\bar{\mathcal{P}}_{\ell}$ is a compact metric space. We deduce from above, that if we denote by $\mathbf Y^n$  a point process with the law of $\Z^n(1)$ conditioned on $\Z^n(1)\neq \delta_0$, there is the bound
 $$\left| \E\left( \varphi\circ  K_{\ell}(\boldsymbol{\zeta}^n(\x)) - \varphi(\x)\right) - rn^{-1} \sum_{j=1}^k \E \left( \varphi\circ  K_{\ell}\left(  \x^*_j \sqcup(x_j+\mathbf Y^n)\right) - \varphi(\x)\right)\right| \leq c n^{-2},$$
 where $c>0$ is some constant which depends on $\ell$ and $\varphi$, but not on $\x$.
 
 On the other hand, we readily deduce from the third assumption, Skorokhod's representation theorem, and again the uniform continuity of $\varphi$,   that 
 there exists a sequence $(\varepsilon(n))_{n\in\N}$ converging to $0$ and depending on $\ell$ and $\varphi$ only, such that for all $\x\in \Mf$ and all $j=1, \ldots, k=\crochet{\x,\mathbf 1}$:
  $$\left| \E\left(  \varphi\circ  K_{\ell}\left(  x^*_j \sqcup(x_j+\mathbf Y^n)\right) \right) -\int_{\Mf} \varphi\circ  K_{\ell}\left(  x^*_j \sqcup(x_j+\y)
\right) \Pi(\dd \y)\right| \leq \varepsilon(n).$$

Putting the pieces together, we now see that
 \begin{equation*} \label{E:unifcv}
 \lim_{n\to \infty} n \E\left( \varphi\circ  K_{\ell}(\boldsymbol{\zeta}^n(\x)) - \varphi(\x)\right) = \mathcal A_{\ell}\varphi(\x) \quad \text{uniformly on }\bar{\mathcal{P}}_{\ell}.
  \end{equation*}
We conclude from Theorem 19.28 in \cite{Kall} (see also Theorem 6.5 of Chapter 1 in \cite{EK}) that there is the
 weak convergence on $\calD(\bar{\mathcal{P}}_{\ell})$
\begin{align} \label{E:trimmedconver}
    \lim_{n\to \infty} \left(K_{\ell}(\Z^n(\lfloor tn \rfloor))\right)_{t\geq 0} = (K_{\ell}(\U(t)))_{t\geq 0}.
    \end{align}
    The proof will be complete if, for every $T>0$, we can show that
\begin{align} \label{E:finitetimeconv}
\lim_{n \rightarrow \infty} (\Z^{n}(\lfloor tn \rfloor))_{0 \leq t \leq T} =(\U(t))_{0 \leq t \leq T} 
\end{align}
in the sense of weak convergence on the space $\mathcal{D}([0,T], \bar\Mf)$ of rcll paths from $[0,T]$ to $\bar\Mf$. For this purpose, we fix a continuous functional $\Phi: \mathcal{D}([0,T], \bar\Mf) \rightarrow [0,1]$. Again by the fact that the total mass of $\Z^n(\lfloor tn \rfloor )$ is non-decreasing in $t$, we have inequality 
\begin{align*}
    \Phi \big( (\Z^n( \lfloor tn \rfloor))_{0 \leq t \leq T} \big)  \geqslant   \Phi \big( (K_{\ell}(\Z^n(\lfloor tn \rfloor)))_{0 \leq t \leq T} \big) \mathbf{1}_{\{ K_{\ell}(\Z^n (\lfloor Tn \rfloor)) \in \mathcal{P}_{\ell} \}}. 
\end{align*}
By the continuity of mapping $(\mathbf{x}_t)_{0 \leq t \leq T} \mapsto \mathbf{1}_{\{K_{\ell}(\mathbf{x}_T) \in \mathcal{P}_{\ell}\}}$ on $\mathcal{D}([0,T], \bar\Mf)$ and weak convergence (\ref{E:trimmedconver}), we have
\begin{align*}
 \liminf_{n \rightarrow \infty}  \E \Big( \Phi \big( (\Z^n( \lfloor tn \rfloor))_{0 \leq t \leq T} \big)  \Big) \geqslant \E \big(  \Phi( (K_{\ell}(\U(t))_{0 \leq t \leq T} ) \mathbf{1}_{ \{K_{\ell}(\U(T)) \in \mathcal{P}_{\ell} \}}  \big).
\end{align*}
Sending $\ell$ to infinity gives 
\begin{align*}
 \liminf_{n \rightarrow \infty}  \E \Big( \Phi \big( (\Z^n( \lfloor tn \rfloor))_{0 \leq t \leq T} \big)  \Big) \geqslant \E \big(  \Phi( (\U(t)_{0 \leq t \leq T} )  \big).
\end{align*}
For the upper bound, replace $\Phi$ by $1- \Phi$ in the above reasoning. This shows (\ref{E:finitetimeconv}) and completes the proof. 
\end{proof}

\subsection{Branching-stable processes and their trimmed versions}
In this section, we provide some background on the construction of branching-stable processes in Section 2 of \cite{BCM} and some of their properties. 
Our presentation is tailored for our purposes; beware also that the present notation sometimes differs from that in \cite{BCM}. 

As a quick summary, we start from a self-similar measure $\Lambda^*$ on a space of counting measures $\x$ on 
$\R_+^*=(0,\infty)$, where self-similarity means that the push-forward image of $\Lambda^*$ by the map $\x \mapsto c\x$ is $c^{-\alpha} \Lambda^*$ for every $c>0$. 
 The atoms of a Poisson point process $\NN$ with intensity $\dd t \otimes \Lambda^*(\dd \x)$ yield a point process $\W(1)$ 
on the upper quadrant $(0,\infty)^2$ that describe the progeny of an immortal and  motionless ancestor located at $0$. More precisely, 
an atom at $(t,x)$ of  $\W(1)$ is interpreted as a birth event occurring at time when the ancestor has age $t$ with the newborn child located at  distance $x$ to the right of the ancestor. We iterate for the next generation, just as for general branching processes \cite{Jagers} by considering $\W(1)$ as the first generation of a branching random walk $(\W(n))_{n\geq 0}$ on $\R_+\times \R_+$ started from $\W(0)=\delta_{(0,0)}$. A  branching-stable process $\S(t)$ at time $t\geq 0$ 
then arises by restricting $\bigsqcup_{n\geq 0} \W(n)$ (i.e. the family of all the atoms appearing in the branching random walk $(\W(n))_{n\geq 0}$, possibly repeated according to their multiplicities) to the strip $[0,t]\times \R_+$. 

We denote the space of locally finite counting measures on $\R_+$ by $\calM$,  endowed with the topology of vague convergence and its Borel sigma-algebra. Possible atoms at $0$ play a special role, and it is convenient to introduce also the notation $\calM^*$ for the subspace of non-zero counting measures $\x\in \calM$ with no atom at $0$. Just as in the Introduction, we write $(x_j)_{j\geq 1}$
for the ordered sequence of the atoms of $\x\in \calM^*$, i.e. $\x=\sum_{j\geq 1} \delta_{x_j}$, where $x_j\in (0,\infty]$ for all $j\geq 1$ and $x_1<\infty$. 
We also write 
$$\calM^1=\{\x\in \calM^*: x_1=1\},$$
so that any $\x\in \calM^*$ has a ``polar'' representation in the form $\x=r\y$ with $r=x_1\in \R^*_+$ and $\y\in \calM^1$.

Let $\alpha >0$; we first consider some finite measure $\boldsymbol{\lambda}$ on $\calM^1$ such that 
\begin{equation}
 \label{condlambda}
 \int_{\calM^1}  \crochet{\y, \bullet^{-\alpha}} \boldsymbol{\lambda}(\dd \y)<\infty, \quad \text{where } \crochet{\y, \bullet^{-\alpha}}= \sum_{j\geq 1} y_j^{-\alpha}.
\end{equation}
We then define a self-similar measure $\Lambda^*$ on $\calM^*$ as the image of $ r^{\alpha -1} \dd r \otimes 
\boldsymbol{\lambda}(\dd \y)$ by the map $(r,\y)\mapsto r\y$. In other words, for every measurable functional $\varphi: \calM^*\to \R_+$,
\begin{equation}
 \label{defLambda}
 \int_{\calM^*} \varphi(\x) \Lambda^*(\dd \x)= \int_0^{\infty} r^{\alpha-1} \int_{\calM^1} \varphi (r \y ) \boldsymbol{\lambda}(\dd \y) \dd r .
\end{equation}

We next introduce a Poisson point process $\NN$ on $(0,\infty)\times \calM^*$ with intensity $\dd t\times \Lambda^*(\dd \x)$ and represent  each atom $(t,\x)$ of $\NN$ as a sequence of atoms $((t,x_j))_{j\geq 1}$ on the fiber $\{t\}\times (0,\infty]$. Discarding  as usual atoms $(t,\infty)$ if any, this induces a point process $\W(1)$ on the quadrant $(0,\infty)^2$ such that 
$$\crochet{\W(1),f}\coloneqq \int_{(0,\infty)\times \calM^*_+} \crochet{\x, f(t,\bullet)} \NN(\dd t, \dd \x),$$
where $f:(0,\infty)^2\to \R_+$ is a generic measurable function. In turn  $\W(1)$ inherits the scaling property, namely, for every $c>0$, 
its image by the map $(t,x)\mapsto (c^{\alpha}t, c x)$ has the same distribution as $\W(1)$.  

We consider $\W(1)$ as the first generation of a branching random walk $(\W(n))_{n\geq 0}$ on $\R_+^2$ started as usual  from a single atom at the origin.   
Finally, for every $t\geq 0$, we write $\S(t)$ for the point process on $\R_+$ defined by
\begin{equation} \label{def_S}
\crochet{\S(t), g}= \sum_{n=0}^{\infty}\crochet{\W(n), g_t}
\end{equation}
where $g:\R_+\to \R_+$ stands for a generic measurable function and $g_t(s,x)=\mathbf 1_{[0,t]}(s) g(x)$. According to Theorem 2.1 in \cite{BCM}, $(\S(t))_{t\geq 0}$ is a branching stable process, that is a  branching process in continuous time which is self-similar with exponent $-\alpha$, in the sense that for every $c>0$, the processes $(\S(c^{-\alpha} t))_{t\geq 0}$ and $(c\S(t))_{t\geq 0}$ have the same law. 

By construction, the law of the branching stable process $\S=(\S(t))_{t\geq 0}$ is determined by the self-similar measure $\Lambda^*$ on the space $\calM^*$. It is convenient to introduce a closely related measure $\Lambda$, now on the space $\calM$, which
is given by the push forward of $\Lambda^*$ by the map $\calM^*\to \calM$, $\x\mapsto \x=\delta_0 \sqcup \x$ that adds an atom at $0$ to $\x$.
Plainly $\Lambda$ is also self-similar and obviously determines $\Lambda^*$. One calls  $\Lambda$ the L\'evy measure of $\S$ as it  bears the same relation to $\S$ viewed as a branching L\'evy process as the classical L\'evy measure to a stable subordinator; see \cite{BerMal}.

 We shall now recall a trimming transformation,  which was introduced more generally  in 
 \cite{BerMal}  for so-called branching L\'evy processes under the name censoring, and  which allows us to represent $(\S(t))_{t\geq 0}$ as the increasing limit of a sequence of Uchiyama branching processes.
 Trimming is better understood when one recalls the interpretation of the construction of $\S(t)$ as a general branching process that we sketched at the beginning of this section. Fix some threshold $b>0$ and write $\W^{[b]}(1)$ for the point process  obtained from $\W(1)$ by restriction to the strip $\R_+\times [0,b]$; in other words, we delete all the atoms $(t,x)$ with $x>b$. Just as above, we see $\W^{[b]}(1)$ as the first generation of a branching random walk $(\W^{[b]}(n))_{n\geq 0}$ on $\R_+^2$, and define
 \begin{equation} \label{def_S(b)}
\crochet{\S^{[b]}(t), g}= \sum_{n=0}^{\infty}\crochet{\W^{[b]}(n), g_t}. \end{equation}
 In words, the trimmed process  $(\S^{[b]}(t))_{t\geq 0}$ is obtained from $\S$  by killing at every birth event every child born at distance greater than $b$ from its parent, of course together with its descent.

We write $\x\mapsto \x^{[b]}=\mathbf 1_{[0,b]}\x$ for  the cut-off map\footnote{Beware that the trimmed process $\S^{[b]}$ is not the image of $\S$ by the cut-off map; the latter would rather be denoted by $(\S(t)^{[b]})_{t\geq 0}$.}
from $\calM$ to the space of finite counting measures $\Mf$ on $[0,\infty)$, and  $\Lambda^{[b]}$ for the image of the L\'evy measure  $\Lambda$ restricted to point processes $\x\in \calM$ having two or more atoms on $[0,b]$ (recall that by construction $\x$
has exactly one  atom at $0$, $\Lambda$-almost everywhere) by this  map. Note that for $\x\in \calM^*$, $\x^{[b]}=\varnothing$ is the zero point mass  if and only if $x_1>b$, and therefore
$$\Lambda^{[b]}(\calM)= \Lambda^*(\x\in \calM^*: x_1\leq b)= \alpha^{-1}b^{\alpha} \boldsymbol{\lambda}(\calM^1)< \infty.$$
Hence $\Lambda^{[b]}$ is a finite measure on $\Mf$; it can  be viewed as the reproduction rate of some 
 Uchiyama branching process. We also stress that $\Lambda^{[b]}$ gives no mass to the zero point process $\varnothing$ and is carried by the subspace of finite counting measures having a single atom at $0$.

 \begin{lemma} \label{L5} The  branching stable process trimmed at threshold $b$,
 $\S^{[b]}=(\S^{[b]}(t))_{t\geq 0}$, of a branching stable process  $\S$ with L\'evy measure $\Lambda$,  is a Uchiyama branching process on $\R_+$ with reproduction rate $\Lambda^{[b]}$. Moreover, the latter fulfills \eqref{E:finiteint}. \end{lemma}
\begin{proof} The statement is closely related to Section 5.2 in \cite{BerMal}; we shall nonetheless present an independent proof.
The trimmed process $\S^{[b]}$ is a general branching process with progeny described by the cut-off version $\NN^{[b]}$ of the Poisson point process $\NN$. By the image property of Poisson point processes, the latter is a Poisson point process on
$\R_+\times \Mf$ with intensity $\dd t\otimes \Lambda^{*[b]}(\dd \x)$, where $\Lambda^{*[b]}$ denotes
the push forward measure of $\Lambda^*$ restricted to $\{ \x\in \calM^*: x_1\leq b\}$ by the cut-off map $\x \mapsto  \x^{[b]}$. 

As a consequence, the first instant at which the ancestor gives birth, 
$$\tau^{[b]}=\inf\{t>0: \NN^{[b]}([0,t]\times \Mf)>0\},$$ has the exponential distribution with parameter $\alpha^{-1}b^{\alpha} \boldsymbol{\lambda}(\calM^1)$
and its offspring at that time has the normalized  law $\alpha b^{-\alpha} \boldsymbol{\lambda}(\calM^1)^{-1}\Lambda^{*[b]}$, independently of $\tau^{[b]}$. Moreover, the point process of the progeny shifted at time  is again a Poisson point process with intensity $\dd t\otimes \Lambda^{*[b]}(\dd \x)$, and is independent of the preceding quantities.

Imagine now that we decide to kill the ancestor at time $\tau^{[b]}$ and simultaneously add a child located at $0$ to its progeny.
Since the new child is born at the same location as the ancestor and precisely at the time when the ancestor is killed, this does not affect what so ever the process itself. This  should make clear the claim that $\S^{[b]}$ evolves like a Uchiyama branching process with reproduction rate given 
the image of $\Lambda^{*[b]}$ by the map $\x\mapsto \delta_0\sqcup \x$ which adds an atom at $0$ to the finite counting measure $\x$. The latter is precisely $\Lambda^{[b]}$. 

Finally, the cut-off $\x \mapsto \x^{[b]} $ preserves atoms at $0$, so $\Lambda^{[b]}$ verifies the first assertion of \eqref{E:finiteint}. The second assertion is plain from 
the  Markov inequality $\crochet{\x^{[b]},\mathbf 1}\leq b^{\alpha}  \crochet{\x, \bullet^{-\alpha}}$ and \eqref{condlambda}.
\end{proof}

\section{A branching-stable limit theorem}
The purpose of this section is to establish our main result which has been presented only informally in the Introduction; let us briefly recall the setting. 
We consider a branching random walk $\Z=(\Z(n))_{n\geq 0}$ on $\R_+$ started from $\Z(0)=\delta_0$. We suppose that $\Z(1)$ has a single atom at the origin a.s. and write 
$$0< X_1\leq X_2 \leq \ldots \leq \infty$$
for the ordered sequence of positive locations of atoms of $\Z(1)$ repeated according to their multiplicities, and then set
$$\Z^*(1)= \Z(1)-\delta_0=\sum_{j\geq 1} \delta_{X_j}.$$
 We shall also need the following basic fact.
\begin{lemma} \label{L4} Assume \eqref{E:C1}, \eqref{E:C2} and \eqref{E:C3}. Then
the limit distribution in  \eqref{E:C2} is the law of a point process that can be expressed in the form
$V^{1/\alpha}\mathbf Y$, with $\mathbf Y$ a point process in $\calM^1$
 and $V$ a uniform random variable on $(0,1)$ independent of $\mathbf Y$. 
Furthermore, the distribution $\boldsymbol{\rho}$ of $\mathbf Y$ satisfies \eqref{condlambda}. 
\end{lemma}
\begin{proof} The first part of the claim belongs to the folklore on multidimensional regular variation; see for instance Chapter 5 in  \cite{Resnick}. For the reader's convenience, we recall the argument. 

Denote the point process arising as weak limit in  \eqref{E:C2} by $\mathbf L$, and write $L_1$ for the location of its left-most atom.
We deduce immediately from \eqref{E:C1} that $V=L_1^{\alpha}$ has the uniform distribution on $[0,1]$, and $\mathbf Y= L_1^{-1} \mathbf L$
is a point process in $\calM^1$. So $\mathbf L= V^{1/\alpha} \mathbf Y$ and we shall now argue that $V$ and $\mathbf Y$ are independent.

We get from the continuous mapping theorem and \eqref{E:C2} that as $t\to 0+$, the conditional distribution of the pair
$$\left( X_1/t, X_1^{-1} \Z^*(1)\right)= \left( X_1/t, (t/X_1)(t^{-1} \Z^*(1))\right) $$ 
conditioned on $X_1\leq t$ converges weakly to that of $(V^{1/\alpha},\mathbf Y)$. For  any continuous and bounded function $f: \calM^1 \to \R$ and $u\in(0,1]$, one has 
$$\lim_{t\to 0+} \E\left( f(X_1^{-1} \Z^*(1)) \ind{X_1/t\leq u} \mid {X_1\leq t}\right) = \E\left( f(\mathbf Y)\ind{V\leq u^{\alpha}}\right).$$
We can also write
\begin{align*}&\E\left( f(X_1^{-1} \Z^*(1))  \ind{X_1/t\leq u} \mid {X_1\leq t}\right)\\
&= \E\left( f(X_1^{-1} \Z^*(1))  \ind{X_1\leq ut} \mid {X_1\leq ut}\right)  \P\left( X_1\leq ut \mid {X_1\leq t}\right).
\end{align*}
So letting $t\to 0$, we get 
$$\E\left( f( \mathbf Y) \ind{V\leq u^{\alpha}}\right)= \E\left( f(\mathbf Y) \right) u^{\alpha},$$
which shows that $V$ and $\mathbf Y$ are independent.

Finally, recall that  $(a_n)_{n\geq 1}$ is a sequence of positive real numbers such that $\P(X_1\leq a_n)\sim 1/n$ and that $\psi$ is the log-Laplace transform of the intensity measure of $\Z(1)$. We deduce  from  \eqref{E:C3} and the bound
$$ \E\left(\sum_{j=1}^{\infty} \exp(-a_n^{-1} X_j) \mid X_1\leq a_n \right) \leq \frac{1}{\P(X_1\leq a_n)} \left( \e^{\psi(1/a_n)}-1\right),
$$
that 
$$\limsup_{n\to \infty} \sum_{j=1}^{\infty} \E\left(\exp(-a_n^{-1} X_j) \mid X_1\leq a_n \right) <\infty.$$
We write $\mathbf Y=\sum_{j\geq 1} \delta_{Y_j}$, and recall  from above that for every $j\geq 1$, the conditional law of $a_n^{-1} X_j$ given $X_1\leq a_n$ converges weakly to that of $Y_j$ on $[1,\infty]$. Fatou's lemma now entails that
$$ \alpha \sum_{j=1}^{\infty} \E\left(\int_0^1 \e^{-t Y_j} t^{\alpha-1} \dd t\right)= \E\left( \sum_{j=1}^{\infty} \exp(-V^{1/\alpha} Y_j) \right)< \infty.$$
Since there is some  $c_{\alpha}>0$ such that 
$$ c_{\alpha} y^{-\alpha} \leq \int_0^1 \e^{-t y} t^{\alpha-1} \dd t \quad\text{ for all  }y\geq 1,$$ 
our last claim follows. 
\end{proof}

Next, we introduce for every $r>0$ the space $\mathcal M_{r,f}$ of counting measures $\x$ on $\R_+$
such that $\crochet{\x, \e^{-r\bullet}}<\infty$. We associate to each $\x\in \mathcal M_{r,f}$
the finite measure $m_{r,\x}$ on $\R_+$ which has the density $\e^{-r\bullet}$ with respect to $\x$.
In words, assuming for simplicity that the counting measure $\x$ is simple, $m_{r,\x}$ is a purely atomic measure, the locations of its atoms are the same as for $\x$, and the mass of an atom at $x$ is $\e^{-r x}$. We then define $d_r(\x,\y)$ for $\x,\y\in \mathcal M_{r,f}$ as the L\'evy-Prokhorov distance between 
$m_{r,\x}$ and $m_{r,\y}$; this makes of $\mathcal M_{r,f}$ a locally compact metric space.
We write $\calD(\mathcal M_{r,f})$ for the space of rcll functions $\omega: \R_+\to \mathcal M_{r,f}$, endowed with the Skorohod $J_1$ topology. 

We may now state rigorously the main result of this work:
\begin{theorem}\label{T1} Assume \eqref{E:C1}, \eqref{E:C2} and \eqref{E:C3}, and 
let $\S$ be a branching stable process with L\'evy measure $\Lambda$,
such that $\Lambda^*$ given by \eqref{defLambda}
 for  $\boldsymbol{\lambda}= \alpha \cdot \boldsymbol{\rho }$ and 
$\boldsymbol{\rho}$ the probability measure on $\calM^1$ that arises in Lemma \ref{L4}. 

Then for every $r>0$, we have
$$\lim_{n\to \infty} \left(a_n^{-1}\Z(\lfloor tn \rfloor)\right)_{t\geq 0} = (\S(t))_{t\geq 0}$$
in the sense of  weak convergence on $\calD(\mathcal M_{r,f})$.
\end{theorem}

\begin{remark}\label{R:1}  Assume \eqref{E:C1} holds; then we have  by a Tauberian theorem:
$$\E(\e^{-tX_1}) \sim  \Gamma(1+\alpha) F_1(1/t)\quad \text{as }t\to \infty.$$
Since $\E(\e^{-tX_j})\leq \E(\e^{-tX_1})$ for all $t> 0$ and all $j\geq 1$, we see that \eqref{E:C3} follows from \eqref{E:C1} whenever the total mass  of 
$\Z(1)$ is bounded, say $X_j=\infty$ a.s. whenever $j\geq k$.
 Indeed, we  then have  
$$\psi(t) \leq  \log(1+k\E(\e^{-tX_1}))\sim k \Gamma(1+\alpha) F_1(1/t).$$
Moreover, if  $\Z(1)$ has actually at most two atoms a.s., i.e. $k=1$ above, 
then \eqref{E:C2} also holds, and more precisely the limiting distribution there is that of the sequence $(V^{1/\alpha}, \infty, \infty, \ldots)$  with $V$ a uniform random  variable on $[0,1]$.
\end{remark}

Throughout the rest of this section, we assume without further mention that \eqref{E:C1}, \eqref{E:C2} and \eqref{E:C3} hold,  and we shall further use the notation in Lemma \ref{L4} and  Theorem \ref{T1}. 
 We first set some further notation relevant to the proof of Theorem \ref{T1}. 
 
 For every $n\geq 1$ and $b>0$, we introduce the rescaled branching random walk $\Z^{[n,b]}=\left(\Z^{[n,b]}(k)\right)_{k\geq 0}$ that results from $\Z$ by first rescaling with a factor $a_n^{-1}$ and then trimming at threshold $b$ (i.e. the children born at distance greater than $b$ from their parents are killed). In words, the first generation is given by
 $$\Z^{[n,b]}(1)=\left(a_n^{-1}\Z(1)\right)^{[b]}= \delta_0 + \sum_{j\geq 1} \mathbf 1_{[0,b]} \delta_{X_j/a_n}.$$
 We start by checking that for every fixed $b>0$, this sequence of branching random walks fulfills the assumptions of Lemma \ref{L2}. The first assumption there is straightforward and we focus on the second and third.
 
 \begin{lemma}\label{L6} We have:
 $$\P(\Z^{[n,b]}(1) \neq \delta_0) \sim b^{\alpha}/n \qquad \text{as }n\to \infty$$
 and, in the notation introduced above Lemma \ref{L5},
 $$\lim_{n\to \infty} \P(\Z^{[n,b]}(1)\in \cdot \mid \Z^{[n,b]}(1)\neq \delta_0)= \Pi^{[b]}(\cdot) $$ 
 in the sense of weak convergence for distributions on $\Mf$,
 where  $\Pi^{[b]}$ denotes the law of the finite point process 
 $$\delta_0\sqcup b(V^{1/\alpha} \mathbf Y)^{[1]} = \delta_0 + \sum_{j: V^{1/\alpha} Y_j\leq 1} \delta_{b V^{1/\alpha} Y_j}.$$
 
 \end{lemma}
\begin{proof} The  events $\{\Z^{[n,b]}(1) \neq \delta_0\}$ and $\{X_1 \leq a_n b\}$ coincide, and  the first estimate is then plain from \eqref{E:C1}.
Next, write $0<X_1^{[n,b]}\leq X_2^{[n,b]}\ldots \leq \infty$ for the ordered sequence of atoms of $\Z^{[n,b]}(1)$ (discarding as usual  the atom at the origin), and 
 fix some $\ell \geq 1$ and $x_j\in[0,b]$ for $j=1, \ldots, \ell$. We then write
 \begin{align*}
& \P(X_1^{[n,b]}\leq x_1, \ldots, X_{\ell}^{[n,b]}\leq x_{\ell}, X_{\ell+1}^{[n,b]}=\infty) \\
=\, & \P(X_1\leq a_n x_1, \ldots, X_{\ell}\leq a_n x_{\ell}, X_{\ell+1} > a_n b) \\
=\, & \P\left(\frac{X_1}{a_n b}\leq \frac{ x_1}{b}, \ldots, \frac{X_{\ell}}{a_n b}\leq \frac{ x_{\ell}}{b},  \frac{X_{\ell+1}}{a_n b}>1 \mid X_{1}\leq a_n b \right )
 \P(X_{1}\leq a_n b).
 \end{align*}
 Recall on the one hand from \eqref{E:C1}  that $\P(X_{1}\leq a_n b) \sim b^{\alpha}/n$, and on the other hand, by Lemma \ref{L4}, that the first term of the product in the last displayed quantity
 converges as $n\to \infty$ to
 $$\P(b V^{1/\alpha}Y_1 \leq x_1, \ldots , b V^{1/\alpha}Y_{\ell} \leq x_{\ell},  V^{1/\alpha}Y_{\ell+1} >1).$$
 This entails our second claim. 
\end{proof}

Lemma \ref{L6} immediately entails the following version of Theorem \ref{T1} for the trimmed processes.

\begin{corollary}\label{C1} Under the same assumptions and notation as in Theorem \ref{T1}, we have for every $b>0$
$$\lim_{n\to \infty} \left(\Z^{[n,b]}(\lfloor tn \rfloor)\right)_{t\geq 0} = (\S^{[b]}(t))_{t\geq 0}$$
in the sense of  weak convergence on $\calD(\Mf)$.
\end{corollary}

\begin{proof} It suffices to observe first the easy identity $\alpha \cdot \Pi^{[b]}= \Lambda^{[b]}$, and then to combine Lemmas \ref{L2}, \ref{L5} and \ref{L6}.
\end{proof}

We can now complete the proof of Theorem \ref{T1}.
\begin{proof}[Proof of Theorem \ref{T1}] Recall  that if $m$ and $m'$ are two measures with $m\geq m'$,
then the L\'evy-Prokhorov distance between $m$ and $m'$ is bounded from above by the total mass of the positive measure $m-m'$. 
Recalling also the definition of the distance $d_r$ on $\mathcal{M}_{r,f}$, this yields for every fixed $b>0$ the bound
$$d_r(\S(s),\S^{[b]}(s)) \leq \crochet{\S(s)-\S^{[b]}(s),\e^{-r\bullet}} \qquad \hbox{ for all }s\geq 0.$$
Moreover, the map  $s\mapsto (\S(s)-\S^{[b]}(s))$ with values in the space of positive measures, is non-decreasing. Working on the time-interval $[0,t]$ for some fixed $t>0$, we get
$$\sup_{0\leq s \leq t} d_r(\S(s),\S^{[b]}(s)) \leq \crochet{\S(t)-\S^{[b]}(t),\e^{-r\bullet}}.$$
On the other hand, we have plainly $\mathbf 1_{[0,b]}\S(t)\leq \S^{[b]}(t)$, so $\S(t)-\S^{[b]}(t)\leq \mathbf 1_{(b,\infty)}\S(t)$, and 
since $\crochet{\S(t),\e^{-r\bullet}})< \infty$ a.s. (see, e.g., Proposition 3.1(ii) in \cite{BCM}),
we conclude that
\begin{equation}\label{E:P1}\lim_{b\to \infty}\sup_{0\leq s \leq t} d_r(\S(s),\S^{[b]}(s)) = 0 \qquad\text{a.s.}
\end{equation}

From \eqref{E:P1}, Corollary \ref{C1}, and the fact that the Prokhorov distance dominates $d_r$ on $\Mf$,
we can find a sequence $(b_n)$ of positive numbers which grows to $\infty$ sufficiently slowly, such that
\begin{equation}\label{E:P2}\lim_{n\to \infty} \left(\Z^{[n,b_n]}(\lfloor sn \rfloor)\right)_{0\leq s \leq t} = (\S(s))_{0\leq s \leq t}
\end{equation}
in the sense of  weak convergence on the space $\calD([0,t], \mathcal{M}_{r,f})$ of rcll paths from $[0,t]$ to $\mathcal{M}_{r,f}$.

By the same argument as in the first paragraph of the proof, we have also for each $n\geq 1$
$$\sup_{0\leq s \leq t} d_r(a_n^{-1}\Z(\lfloor ns\rfloor),\Z^{[n,b_n]}(\lfloor ns\rfloor)) \leq \crochet{a_n^{-1}\Z(\lfloor nt\rfloor),\mathbf 1_{(b_n,\infty)}\e^{-r\bullet}}.$$
Thanks to the Markov inequality,  the right-hand side is bounded from above by 
$$\e^{-b_n(r-\tilde{r})}\crochet{a_n^{-1}\Z(\lfloor nt\rfloor),\e^{-\tilde{r}\bullet}} = \e^{-b_n(r-\tilde{r})} \crochet{\Z(\lfloor nt\rfloor),\e^{-a_n^{-1} \tilde{r} \bullet}}$$
where $0<\tilde{r}<r$. Recall that $\psi$ denotes the log-Laplace transform of the intensity measure of $\Z(1)$, so the expectation of the right-hand side equals
$$\E\left(\crochet{\Z(\lfloor nt\rfloor),\e^{-a_n^{-1} \tilde{r}\bullet}}\right) = \exp\left(\lfloor nt\rfloor \psi(\tilde{r}/a_n)\right).$$
This quantity remains bounded as $n\to \infty$ by assumption \eqref{E:C3} and the fact that $(a_n)$ is regularly varying. Putting the pieces together, we have shown that
\begin{equation}\label{E:P3}\lim_{n\to \infty}  \E\left( \sup_{0\leq s \leq t} d_r(a_n^{-1}\Z(\lfloor ns\rfloor),\Z^{[n,b_n]}(\lfloor ns\rfloor))\right)=0.
\end{equation}
Applying the argument of Lemma VI. 3.31 on page 352 in \cite{JS} in the setting of metric spaces rather than $\R^d$, we conclude from \eqref{E:P2} and \eqref{E:P3} that
\begin{equation}\label{E:P4}\lim_{n\to \infty} \left(a_n^{-1}\Z(\lfloor sn \rfloor)\right)_{0\leq s \leq t} = (\S(s))_{0\leq s \leq t}
\end{equation}
in the sense of weak convergence on $\calD([0,t], \mathcal{M}_{r,f})$, and the proof is complete.
\end{proof}

 \begin{appendix}
 \section {Appendix: Some spaces of counting measures}
 We list below the notation for several spaces of counting measures which appear in this text.
\begin{itemize}
\item $\calM$ denotes the space of locally finite counting measures on $\R_+$ equipped the topology of vague convergence and its Borel sigma-algebra.
\item $\Mf$ denotes the space of finite counting measures (first on $\R^d$ in Section \ref{S:Uch}, and then on $\R_+$ in the rest of the article), endowed with the L\'evy-Prokhorov distance.
\item $\bar \Mf= \Mf\cup\{\partial\}$ is the one-point compactification of $\Mf$ for the L\'evy-Prokhorov metric.
\item $\calM_{\ell}$, for some $\ell\geq 1$, denotes the space of counting measures in $\Mf$ with total mass at most $\ell$ (first on $\R^d$ in Section \ref{S:Uch}, and then on $\R_+$ in the rest of the article),  endowed with the L\'evy-Prokhorov distance.
\item $\calM^*$ denotes the subspace of non-zero counting measures in $\calM$ with no atom at $0$.
\item $\calM^1$ is the subspace of counting measures in $\calM^*$ with left-most atom located at $1$.
\item $\mathcal M_{r,f}$, for some $r>0$, denotes the subspace of counting measures $\x\in \calM$ with $\crochet{\x, \e^{-r\bullet}}<\infty$.
\end{itemize}

\end{appendix}

\bibliography{StBRW.bib}

\end{document}